\documentclass{amsart}
\usepackage{amssymb,amsthm,amstext,amsmath,amscd,latexsym,amsfonts,amscd}
\usepackage[usenames]{color}
\usepackage{setspace}
\usepackage{comment}
\usepackage{enumerate}
\usepackage{braket}
\usepackage[all]{xy}
\usepackage{multirow}
\usepackage{mathrsfs}
\usepackage{graphicx}
\usepackage{enumitem}
\newtheorem{theorem}[equation]{Theorem}
\newtheorem{proposition}[equation]{Proposition}
\newtheorem{lemma}[equation]{Lemma}
\newtheorem{corollary}[equation]{Corollary}

\theoremstyle{definition}

\newtheorem{remark}[equation]{Remark}
\newtheorem{notation}[equation]{Notation}

\newtheorem{conjecture}[equation]{Conjecture}
\newtheorem{question}[equation]{Question}
\theoremstyle{plain}
\newcommand{\jump}[1]{\ensuremath{[\![#1]\!]} }
\newcommand{\disp}{\displaystyle}

\newcommand{\Spec}{\mathop{\mathrm{Spec}}\nolimits}
\newcommand{\Supp}{\mathop{\mathrm{Supp}}\nolimits}

\newcommand{\supp}{\mathop{\mathrm{supp}}\nolimits}
\newcommand{\cosupp}{\mathop{\mathrm{cosupp}}\nolimits}

\newcommand{\Ext}{\mathop{\mathrm{Ext}}\nolimits}

\newcommand{\Hom}{\mathop{\mathrm{Hom}}\nolimits}

\newcommand{\RHom}{\mathop{\mathrm{RHom}}\nolimits}
\newcommand{\RGamma}{\mathop{\mathrm{R\Gamma}}\nolimits}
\newcommand{\LLambda}{\mathop{\mathrm{L\Lambda}}\nolimits}
\newcommand{\Mod}{\mathop{\mathrm{Mod}}\nolimits}

\newcommand{\LotimesR}{\mathop{\otimes_R^{\rm L}}\nolimits}

\newcommand{\fa}{\mathfrak{a}}

\newcommand{\fc}{\mathfrak{c}}

\newcommand{\fp}{\mathfrak{p}}

\newcommand{\fq}{\mathfrak{q}}

\newcommand{\fm}{\mathfrak{m}}

\makeatletter
  
  \@addtoreset{equation}{section}
\makeatother
\begin{document}
\title[Cosupports and minimal pure-injective resolutions of affine rings]{Cosupports and minimal pure-injective resolutions of affine rings}
\author[T. Nakamura]{Tsutomu Nakamura}
\address[T. Nakamura]{Dipartimento di Informatica - Settore di Matematica, Universit\`a degli Studi di Verona, Strada le Grazie 15 - Ca' Vignal, I-37134 Verona, Italy}
\email{tsutomu.nakamura@univr.it}
\begin{abstract} 
We prove that any affine ring $R$ over a field $k$ has full cosupport, {\it i.e.}, the cosupport of $R$ is equal to $\Spec R$.
Using this fact, we give a complete description of all terms in a minimal pure-injective resolution of $R$, provided that $|k|=\aleph_1$ and $\dim R\geq 2$, or $|k|\geq \aleph_1$ and $\dim R=2$.
As a corollary, we obtain a partial answer to a conjecture by Gruson.
\end{abstract}

\maketitle
\section{Introduction}
\label{Introduction}
Let $R$ be a commutative noetherian ring.
We denote by $\mathrm{D}(R)$ the unbounded derived category of $R$. Note that the objects of $\mathrm{D}(R)$ are complexes of $R$-modules, which are cohomologically indexed;
$$X=(\cdots\to X^{i-1}\to X^{i}\to X^{i+1}\to \cdots).$$

The (small) support of $X\in \mathrm{D}(R)$ is defined as
$$
\supp_R  X= \{ \fp \in \Spec R \mid \kappa(\fp) \otimes^{\rm L}_R X \neq 0 \},
$$ where $\kappa(\fp)=R_\fp/\fp R_\fp$.
This notion was introduced by Foxby \cite{F}.
Let $\fa$ be an ideal of $R$, and write $\Gamma_{\fa}$ for the $\fa$-torsion functor $\varinjlim_{n\geq 1}\Hom_R(R/\fa^n,-)$ on the category $\Mod R$ of $R$-modules.
Then we have $\supp_R\RGamma_{\fa}X$  $\subseteq V(\fa)$ for any $X\in \mathrm{D}(R)$, see \cite[Proposition 3.13]{SW}.

The cosupport of $X\in \mathrm{D}(R)$ is defined as 
$$\cosupp_R X=\{\fp\in \Spec R \mid \RHom_R(\kappa(\fp),X)\neq 0 \}.$$
Benson, Iyengar and Krause \cite{BIK2} introduced this notion in more general triangulated categories, developing Neeman's work \cite{N2}.
We write $\Lambda^{\fa}$ for the $\fa$-adic completion functor $\varprojlim_{n\geq 1}(R/\fa^n\otimes_R-)$ on $\Mod R$.
Greenlees and May \cite{GM} proved that the left derived functor $\LLambda^{\fa}:\mathrm{D}(R)\to \mathrm{D}(R)$ is a right adjoint to $\RGamma_{\fa}:\mathrm{D}(R)\to \mathrm{D}(R)$, see also \cite[\S 4; p.\,69]{L}.
It follows from the adjointness property that $\cosupp_R\LLambda^{\fa}X\subseteq V(\fa)$ for any $X\in \mathrm{D}(R)$.

If $M$ is a finitely generated $R$-module, then Nakayama's lemma implies that $\supp_RM=\Supp_RM=\{\fp\in \Spec R \mid M_\fp\neq 0 \}$.
In particular, $\supp_RR$ is nothing but $\Spec R$.
However, it is not easy to compute $\cosupp_RR$ in general.

Since $\LLambda^{\fa}R\cong\Lambda^{\fa}R$, the cosupport of $\Lambda^{\fa}R$ is contained in $V(\fa)$.
Hence we have $\cosupp_RR\subseteq V(\fc_R)$, where $\fc_R$ denotes the sum of all ideals $\fa$ such that $R$ is $\fa$-adically complete.
In \cite[Question 6.13]{SW}, Sather-Wagstaff and Wicklein questioned whether the equality $\cosupp_R R = V(\fc_R)$ holds for any commutative noetherian ring $R$ or not.
Thompson \cite[Example 5.6]{PT1} gave a negative answer to this question.
He proved that if $k$ is a field and $R=k\jump{X}[Y]$, then $\cosupp_R R$ is strictly contained in $V(\fc_R)$. Moreover, it was also shown that the cosupport of this ring is not Zariski closed.

Following \cite{PT1}, we say that a commutative noetherian ring $R$ has {\it full cosupport} if $\cosupp_RR$ is equal to $\Spec R$.
Let $k$ be a field and $R$ be the polynomial ring $k[X_1,\ldots, X_n]$ in $n$ variables over $k$.
Then, since $\fc_R=(0)$, we expect that $\cosupp_RR = \Spec R$. 
Indeed, this is true.

\begin{theorem}
\label{main theorem}
Let $k$ be a field and $n$ be a non-negative integer.
The polynomial ring $k[X_1,\ldots,X_n]$ has full cosupport.
\end{theorem}
\color{black}

This fact was known in the case where $n\leq 1$ or $k$ is at most countable, see \cite[Theorem 4.13]{PT1}. See also \cite[Question 6.16]{SW}, \cite[Proposition 4.18]{BIK2} and \cite[Proposition 3.2]{G}.\medskip

We say that $R$ is an affine ring over a field $k$ if $R$ is finitely generated as a $k$-algebra.
By Theorem \ref{main theorem}, we can show the following corollary, which is the main result of this paper.

\begin{corollary}
\label{main cor}
Any affine ring over a field has full cosupport.
\end{corollary}

In Section 2, we prove the two results above.
Section 3 is devoted to summarize some facts about cosupport and minimal pure-injective resolutions. 
Section 4 contains applications of the main theorem. 
Let $k$ be a field with $|k|=\aleph_1$ and $R$ be an affine ring over $k$ such that $\dim R\geq 2$.
We specify all terms of a minimal pure-injective resolution of $R$. 
As a corollary, it is possible to give a partial answer to Gruson's conjecture. 
Suppose that $R$ is a polynomial ring $k[X_1,\ldots, X_n]$ over a field $k$. The conjecture states that $\Ext_R^i(R_{(0)},R)\neq 0$ if and only if $i=\inf\{s+1, n\}$, where $s$ is defined by the equation $|k|=\aleph_s$ if $k$ is infinite, and $s=0$ otherwise. 
We prove that this conjecture is true when $\inf\{s+1, n\}=2$.

\begin{remark}\label{Rickard}
There is another interesting consequence of Corollary \ref{main cor}. Let $R$ be an affine ring over a field. 
It then follows from the corollary and Neeman's theorem \cite[Corollary 2.8]{N2} that $\mathrm{D}(R)$ coincides with the smallest colocalizing subcategory containing $R$ (see also \cite[\S 9]{BIK2}). In other words, $R$ is a cogenerator of $\mathrm{D}(R)$. Compare this fact with  recent work \cite[\S 5]{JR} by Rickard. 
\end{remark}

\noindent{\bf Acknowledgements.}
The author is deeply grateful to Yuji Yoshino for his guidance at Okayama University. 
The author also thanks Srikanth Iyengar and Peder Thompson for their many helpful comments. 
\section{Proof of main result}

We start with the following proposition. 

\begin{proposition}[{Thompson \cite[Theorem 4.6]{PT1}}]\label{Thompson's lemma}
Let $\varphi: R\to S$ be a morphism of commutative 
noetherian rings. 
Suppose that $\varphi$ is finite, i.e., $S$ is finitely generated as an $R$-module.
Let $f:\Spec S\to \Spec R$ be the canonical map defined by $f(\fq)=\varphi^{-1}(\fq)$ for $\fq\in \Spec S$.
Then, there is an equality
$$\cosupp_S S=f^{-1}(\cosupp_R R).$$
In other words, for $\fq\in \Spec S$, we have $\fq\in \cosupp_SS$ if and only if $f(\fq) \in \cosupp_R R$.
\end{proposition}

When $R$ has finite Krull dimension, this fact follows from Lemma \ref{lemma 2}, Proposition \ref{Enochs's theorem} and (\ref{Thompson's isom}). See also Remark \ref{misfortune}.

\begin{remark}\label{Key remark}
Let $R$ be a commutative noetherian ring. We denote by $0_R$ the zero element of $R$. 
The following statements hold by Proposition \ref{Thompson's lemma}.

(i) Let $\fp$ be a prime ideal of $R$. 
Then we have $(0_{R/\fp})\in \cosupp_{R/\fp} R/\fp$ if and only if $\fp\in \cosupp_R R$.

(ii) Let $S$ be an integral domain. 
Suppose that there is a finite morphism $\varphi:R\to S$ of rings such that $\varphi$ is an injection. 
Then we have  $(0_S)\in \cosupp_SS$ if and only if $(0_R)\in \cosupp_R R$.
\end{remark}

We next recall a well-known description of local cohomology via \v{C}ech complexes.
Let $R$ be a commutative noetherian ring and $\fa$ be an ideal of $R$.
Let $\boldsymbol{x}=\{x_1,\ldots,x_n\}$ be a system of generators of $\fa$.
In $\mathrm{D}(R)$, $\RGamma_{\fa}R$ is isomorphic  to the (extended) \v{C}ech complex with respect to $\boldsymbol{x}$;
$$
0\longrightarrow R \longrightarrow \disp{\bigoplus_{1\leq i\leq n} R_{x_i}}\longrightarrow  \disp{\bigoplus_{1\leq i<j\leq n} R_{x_ix_j}}\longrightarrow \cdots \longrightarrow \disp{R_{x_1\cdots x_n}}\longrightarrow 0.
$$
Here, for an element $y\in R$, $R_y$ denotes the localization of $R$ with respect to the multiplicatively closed set $\{1, y, y^2,\ldots \}$. See \cite[Lecture 7; \S 4]{24} for details.

\begin{remark}\label{localization remark}
Let $k$ be a field and $n$ be a non-negative integer.
Set $R=k[X_1,\ldots, X_n]$ and $S=k[X_1,\ldots, X_{n+1}]$.
Take $y\in R$.
Then $S/(1-yX_{n+1})$ is isomorphic to $R_y$ as a $k$-algebra.
\end{remark}

Let $R$ be a commutative noetherian ring.
For $\fp\in \Spec R$, $E_R(R/\fp)$ denotes the injective envelope of $R/\fp$. Moreover, when $R$ is an integral domain, we denote by $Q(R)$ the quotient field of $R$. 

\begin{proof}[Proof of Theorem \ref{main theorem}]
Set $R=k[X_1,\ldots, X_n]$ and take $\fp\in\Spec R$.
We have to prove that $\fp\in \cosupp_RR$.
By Remark \ref{Key remark}(i), 
this is equivalent to showing that $(0_{R/\fp})\in \cosupp_{R/\fp}R/\fp$.
In addition, the Noether normalization theorem yields a finite morphism $\varphi:k[X_1,\ldots, X_m]\to R/\fp$ of rings such that $\varphi$ is an injection, where $\dim R/\fp=m$.  
Therefore, by Remark \ref{Key remark}(ii), it is sufficient to show that $(0_R)\in \cosupp_R R$ for any $n\geq 0$.

We assume that $(0_R)\notin \cosupp_R R$ for some $n\geq 0$, and deduce a contradiction.
Let $\fp$ be a prime ideal of $S=k[X_1,\ldots,X_{n+1}]$ with $\dim S/\fp=n$.
The Noether normalization theorem yields a finite morphism $\psi:R\to S/\fp$ of rings such that $\psi$ is an injection. 
Hence we have $(0_{S/\fp})\notin \cosupp_{S/\fp} S/\fp$ by the assumption and Remark \ref{Key remark}(ii).
Consequently, for any $y\in R$, it follows from Remark \ref{localization remark} that $(0_{R_y})\notin \cosupp_{R_y}R_y$.
In other words, we have $\RHom_{R_y}(Q(R_y), R_y)=0$ in $\mathrm{D}(R_y)$.
This implies that $(0_R)\notin \cosupp_R R_y$ for any $y\in R $, since $\RHom_{R_y}(Q(R_y), R_y)\cong \RHom_{R}(Q(R), R_y)$ in $\mathrm{D}(R)$.

Now set $\fm=(X_1,\ldots,X_n) \subseteq R$. 
Then $\RGamma_{\fm}R$ is isomorphic to the \v{C}ech complex with respect to $\boldsymbol{x}=\{X_1,\ldots,X_n\}$.
Hence we have $\RHom_R(Q(R),\RGamma_{\fm}R)=0$ by the above argument. 
However, there is an isomorphism $\RGamma_{\fm}R\cong E_R(R/\fm)[-n]$ in $\mathrm{D}(R)$, see \cite[Theorem 11.26]{24}. Moreover, the canonical map $R\to R/\fm$ induces a non-trivial map $Q(R)\to E_R(R/\fm)$, since $E_R(R/\fm)$ is injective.
Therefore $\RHom_{R}(Q(R), \RGamma_{\fm}R)$ must be non-zero in $\mathrm{D}(R)$. This is a contradiction.
\end{proof}

\begin{proof}[Proof of Corollary \ref{main cor}]
If a commutative noetherian ring $R$ has full cosupport, Proposition \ref{Thompson's lemma} implies  that $R/\fa$ has full cosupport for any ideal $\fa$ of $R$.
Therefore, this corollary follows from Theorem \ref{main theorem}.
\end{proof}

\begin{remark}
Let $k$ be a field and $R$ be an affine ring over $k$.
Let $y\in R$.
We see from Corollary \ref{main cor} and Remark \ref{localization remark} that $R_y$ has full cosupport.
\end{remark}

\begin{question}\label{question 1}
Let $k$ be a field and $n$ be a non-negative integer. Set $R=k[X_1,\ldots, X_n]$, and let $U$ be a multiplicatively closed subset of $R$.
Does the ring $U^{-1}R$ have full cosupport?
\end{question}

A commutative ring $R$ is said to be essentially of finite type over a field $k$ when $R$ is a localization of an affine ring over $k$.
If the question above is true, then any ring essentially of finite type over $k$ has full cosupport, by Corollary \ref{main cor}.

\begin{remark}
Let $R$ be a commutative noetherian ring with finite Krull dimension and $X\in \mathrm{D}(R)$ be a complex with finitely generated cohomology modules. 
As shown in \cite[Corollary 4.4]{PT1}, there is an equality $\cosupp_RX=\supp_RX\cap \cosupp_RR$, see also \cite[Theorem 6.6]{SW}.
Hence, if $R$ has full cosupport, then we have $\cosupp_RX=\supp_RX$, so that
 the cosupport of any finitely generated $R$-module $M$ is Zariski-closed, since $\supp_RM=\Supp_RM$.
\end{remark}

In \cite[Example 5.7]{PT1}, Thompson proved that the cosupport of $k\jump{X}[Y_1,\ldots, Y_n]$ is not Zariski-closed for $n\geq 1$, where $k$ is any field.
Hence the following question naturally arises.

\begin{question}\label{question 2}
Let $R$ be any commutative noetherian ring.
Is the cosupport of $R$ specialization-closed?
\end{question}

\section{Cosupport and Minimal pure-injective resolutions}


In this section, we summarize some known facts about cosupport and minimal pure-injective resolutions. They will be used in Section 4. 
\medskip

Let $R$ be a commutative noetherian ring.
For an $R$-module $M$ and an ideal $\fa$ of $R$, we denote by $M^\wedge_{\fa}$ the $\fa$-adic completion $\Lambda^{\fa}M=\varprojlim_{n\geq 1} M/\fa^n M$.
In addition, for the localization $M_\fp$ at a prime ideal $\fp$, we also write $\widehat{M_\fp}=\Lambda^{\fp}M_\fp$.

We recall two formulas for an $R$-module of the form $\prod_{\fp\in \Spec R} T_\fp$, where $T_\fp$ is the $\fp$-adic completion of a free $R_\fp$-module for $\fp\in \Spec R$. Take $\fq\in \Spec R$, and write $U(\fq)=\{\fp\in \Spec R\mid\fp\subseteq \fq\}$.
It then holds that 
$$\Hom_R(R_\fq, \prod_{\fp\in \Spec R}T_\fp)\cong \prod_{\fp\in U(\fq)}T_\fp, \ \ \ \ \Lambda^{\fq}(\prod_{\fp\in \Spec R}T_\fp)\cong \prod_{\fp\in V(\fq)}T_\fp,$$
see \cite[Lemma 2.2]{PT2}.

Let $P=(0\to PE^0(M)\to PE^1(M)\to \cdots)$ be a minimal pure-injective resolution of an $R$-module $M$; it is constructed by pure-injective envelopes, see \cite[\S 6.7]{EJ}.
As every term of $P$ is cotorsion (see \cite[Definition 5.3.22]{EJ}), we have $\RHom_R(F, M)\cong \Hom_R(F, P)$ in $\mathrm{D}(R)$ for a flat $R$-module $F$.
Further, if $M$ is flat, then each $PE^i(M)$ is isomorphic to the direct product of the $\fp$-adic completion of a free $R_\fp$-module for $\fp \in \Spec R$, see \cite[\S 8.5]{EJ}.

\begin{notation}\label{minimal pure remark 2}
Let $F$ be a flat $R$-module.
As mentioned above, we may write $PE^i(F)= \prod_{\fp\in \Spec R} T^i_\fp$ for $i\geq 0$, where $T^i_\fp=(\bigoplus_{B_\fp^i}R_\fp)^\wedge_{\fp}$ for some cardinal $B_\fp^i$.
\end{notation}

\begin{lemma}\label{lemma 2}
Suppose that $\dim R$ is finite. 
Let $F$ be a flat $R$-module.
Using Notation \ref{minimal pure remark 2}, we write $PE^i(F)=\prod_{\fp\in \Spec R} T^i_\fp$. 
Then $\fp \in \cosupp_RF$ if and only if $T^i_\fp\neq 0$ for some $i$.
\end{lemma}

For the reader's convenience, we justify the above lemma in the next remark.

\begin{remark}\label{minimal pure remark 3}
Let $F$ and $PE^i(F)$ be as in Lemma \ref{lemma 2}. 

(i) By \cite[Corollary 8.5.10]{EJ}, we have $PE^i(F)= \prod_{\fp\in W_{\geq i}}T^i_\fp$ for $i\geq 0$, where $W_{\geq i}=\{\fp\in \Spec R \mid \dim R/\fp\geq i\}$.

(ii) Let $P$ be a minimal pure-injective resolution of $F$, where $P^i=PE^i(F)$. It then follows from (i) that $P^i=0$ for $i>n=\dim R$.
Moreover, recall that $\LLambda^{\fp}X\cong \Lambda^{\fp}X$ and $\kappa(\fp)\LotimesR X\cong \kappa(\fp)\otimes_R X$ if $X$ is a complex of flat $R$-modules with $X^i=0$ for $i\gg 0$.
Then we see that
\begin{align*}
\LLambda^{\fp}\RHom_R(R_\fp, F)&\cong \Lambda^{\fp}\Hom_R(R_\fp, P)\\
&\cong (0\to  T^0_\fp\to  \cdots \to T^n_\fp\to 0),\\
\kappa(\fp)\LotimesR\RHom_R(R_\fp, F)&\cong \kappa(\fp)\otimes_R\Hom_R(R_\fp, P)\\
&\cong (0\to  \textstyle{\bigoplus_{B_\fp^0}\kappa(\fp)}\to  \cdots \to \textstyle{\bigoplus_{B_\fp^n}\kappa(\fp)}\to 0),
\end{align*}
where $T^i_\fp=(\bigoplus_{B_\fp^i}R_\fp)^\wedge_{\fp}$.
Note that all differentials in 
$\kappa(\fp)\otimes_R\Hom_R(R_\fp, P)$ are zero, see \cite[Proposition 8.5.26]{EJ}. Hence Lemma \ref{lemma 2} follows from the following bi-implications for $X\in \mathrm{D}(R)$;
$$\fp\in \cosupp_RX\Leftrightarrow \LLambda^{\fp}\RHom_R(R_\fp, X)\neq 0\Leftrightarrow \kappa(\fp)\LotimesR \RHom_R(R_\fp, X)\neq 0,$$
see \cite[Proposition 4.4]{SW}.

(iii) The cardinality $B^i_\fp$ defining the $R$-module $T^i_\fp=(\bigoplus_{B_\fp^i}R_\fp)^\wedge_{\fp}$  is nothing but  $\dim_{\kappa(\fp)}\mathrm{H}^i\big(\kappa(\fp)\LotimesR\RHom_R(R_\fp, F)\big)$.
\end{remark}

\begin{proposition}
\label{Enochs's theorem}
Suppose that $\varphi:R\to S$ is a finite morphism of commutative noetherian rings.
Let $P$
be a minimal pure-injective resolution of a flat $R$-module $F$.
Then $S\otimes_RP$ is a minimal pure-injective resolution of $S\otimes_RF$ in $\Mod S$.
\end{proposition}

See \cite[Theorem 8.5.1]{EJ} for the proof.\medskip

Let $\varphi: R\to S$ be as in Proposition \ref{Enochs's theorem} and $f:\Spec S\to \Spec R$ be a canonical map induced by $\varphi$. For later use, we recall an isomorphism in \cite[Theorem 4.6]{PT1}, which is useful to know components of $S\otimes_RP$ in the proposition.
Let $B_\fp$ be some cardinality for each $\fp\in \Spec R$.
Then it holds that 
\begin{align}\label{Thompson's isom}
S\otimes_R\Biggl(\hspace{1pt}\prod_{\fp\in \Spec R}\Biggl(\bigoplus_{B_\fp}R_\fp\Biggr)^\wedge_{\fp}\hspace{1pt}\Biggr)\cong \prod_{\fq\in \Spec S}\Biggl(\bigoplus_{B_\fq}S_\fq\Biggr)^\wedge_{\fq}
\end{align}
where $B_\fq=B_\fp$ if $f(\fq)=\fp$.
\color{black}

\begin{remark}\label{misfortune}
Thompson \cite[Theorem 2.7]{PT1} formulated Lemma \ref{lemma 2} in a more general setting. However there is an error in his theorem, and it is used to show \cite[Theorem 4.6]{PT1} (Proposition \ref{Thompson's lemma}).
What we have to emphasize here is that the error does not influence this paper at all. In fact, when $R$ has finite Krull dimension, the reader can see from this section the validity of Proposition \ref{Thompson's lemma}, as mentioned just after it.

In recent work \cite[\S 5]{NT}, we have corrected and improved \cite[Theorem 2.7]{PT1}. It has been also ensured that \cite[Theorem 4.6]{PT1} is actually true for arbitrary commutative noetherian rings.

\end{remark}


\section{Applications}

In this section, we give a complete description of all terms in a minimal pure-injective resolution of an affine ring $R$ over a field $k$, where $|k|=\aleph_1$ and $\dim R\geq 2$, or $|k|\geq \aleph_1$ and $\dim R=2$.
Consequently we obtain a partial answer to the following conjecture by Gruson, which is stated in a paper \cite{AT} of Thorup.

\begin{conjecture}\label{Gruson conj}
Let $k$ be a field.
We define $s$ by the equation $|k|=\aleph_s$ of  cardinals if $k$ is infinite, and $s=0$ otherwise. 
Let $n$ be a non-negative integer, and set $R=k[X_1,\ldots,X_n]$.
Then $\Ext_R^i(Q(R),R)\neq 0$ if and only if $i=\inf\{s+1, n\}$.\end{conjecture}

If $k$ is at most countable or $n\leq 1$, {\it i.e.}, $\inf\{s+1, n\}\leq 1$, then this conjecture is true, see \cite[\S1; 1]{AT}. 
In this section, we shall prove the following theorem, which implies that the conjecture is true in the case that  $\inf\{s+1, n\}= 2$.

\begin{theorem}\label{partial answer}
Let $k$ be an uncountable field and $R$ be an affine ring over $k$ with $\dim R\geq 2$.
Assume that $|k|=\aleph_1$ or $\dim R=2$.
Let  $\fp$ be a prime ideal of $R$ with $\dim R/\fp\geq 2$.
Then $\Ext^i_R(R_\fp, R)\neq 0$ if and only if $i=2$. 
\end{theorem}

This theorem essentially follows from the main result (Corollary \ref{main cor}) and the following result of Thorup. 

\begin{proposition}[{\cite[Theorem 13]{AT}}]
\label{Thorup's theorem}
Let $k$ be an uncountable field and $n$ be an integer with $n\geq 2$.
Set $R=k[X_1,\ldots,X_n]$.
Then $\Ext^1_R(Q(R),R)=0$. 
\end{proposition}

We start with extending this to the next corollary.

\begin{corollary}\label{weak form}
Let $k$ be an uncountable field.
Suppose that $S$ is an affine domain over $k$ with $\dim S\geq 2$. 
Then $\Ext^1_{S}(Q(S),S)=0$. 
\end{corollary}

\begin{proof}
Set $n=\dim S$ and $R=k[X_1,\dots, X_n]$.
The Noether normalization theorem yields a finite injection $\varphi:R\rightarrow S$.
Write $f:\Spec S\to \Spec R$ for the canonical map induced by $\varphi$.
Let $P$ be a minimal pure-injective resolution of $R$.
Using Remark \ref{minimal pure remark 3}(i) and (iii), we write $P^1= \prod_{\fp\in W_{\geq 1}} T^1_\fp$, 
where $T^1_\fp=(\bigoplus_{B_\fp^1}R_\fp)^\wedge_{\fp}$ and $B^1_\fp=\dim_{\kappa(\fp)}\mathrm{H}^1\big(\kappa(\fp)\LotimesR\RHom_R(R_\fp, R)\big)$.
By Proposition \ref{Thorup's theorem},
we have $\Ext_R^1(Q(R), R)=0$, so that $B^1_{(0_R)}=0$.
In other words, it holds that $$P^1=\prod_{\fp\in W_{\geq 1}} T^1_\fp=\prod_{\fp\in W_{\geq 1}\backslash \{(0_R)\}} T^1_\fp.$$
Recall that $S\otimes_RP$ is a minimal pure-injective resolution of $S$ in $\Mod S$ by Proposition \ref{Enochs's theorem}.
Moreover, since $f(0_S)=(0_R)$, we see from (\ref{Thompson's isom}) that 
$\Hom_S(Q(S),S\otimes_RP^1)=0$.
Hence it holds that 
$$\Ext^1_S(Q(S), S)\cong \mathrm{H}^1\big(\Hom_S(Q(S), S\otimes_RP)\big)=0.$$
\end{proof}

Using this corollary, we can show the next result.

\begin{corollary}\label{strong form}
Let $k$ be an uncountable field and $R$ be an affine ring over $k$.
If $\fp\in \Spec R$ with $\dim R/\fp\geq 2$, then we have 
$$\mathrm{H}^1\big(\kappa(\fp)\LotimesR\RHom_{R}(R_\fp,R)\big)=0.$$
\end{corollary}

\begin{proof}
Notice that $\kappa(\fp)\LotimesR\RHom_{R}(R_\fp,R)\cong R/\fp\LotimesR \RHom_{R}(R_\fp,R)$.
Let $P$ be a minimal pure-injective resolution of $R$.
By Remark \ref{minimal pure remark 3}(ii), $R/\fp\otimes_R\Hom_{R}(R_\fp,P)$ is a complex of modules over $\kappa(\fp)=Q(R/\fp)$. Moreover it is seen from (\ref{Thompson's isom}) that we can also get this complex by applying $\Hom_{R/\fp}(Q(R/\fp),-)$ to $R/\fp\otimes_RP$;
$$R/\fp\otimes_R \Hom_{R}(R_\fp,P)
\cong \Hom_{R/\fp}(Q(R/\fp), R/\fp\otimes_RP).$$
Since $R/\fp\otimes_RP$ is a minimal pure-injective resolution of $R/\fp$ in $\Mod R/\fp$ by Proposition \ref{Enochs's theorem}, it follows that
$$\mathrm{H}^i\big(R/\fp\LotimesR \RHom_{R}(R_\fp,R)\big)
\cong \Ext^i_{R/\fp}(Q(R/\fp), R/\fp).$$
The right-hand side vanishes if $i=1$, by Corollary \ref{weak form}.
\end{proof}

Let $R$ be a commutative noetherian ring. Using Remark \ref{minimal pure remark 3}(i),
we write  $PE^i(R)= \prod_{\fp\in W_{\geq i}} T^i_\fp$.
We set $W_i=\{\fp\in \Spec R \mid \dim R/\fp=i \}$.
It is known that 
$$PE^0(R)=\prod_{\fm\in W_{0}}\widehat{R_\fm},$$ see \cite[Proposition 6.7.3]{EJ}. In general, it is not easy to know non-trivial components of $PE^i(R)$ for  $i>0$. However, in the case of affine rings, we can obtain the following result by Corollary \ref{main cor} and Lemma \ref{lemma 2}.

\begin{corollary}\label{add remark}
Let $R$ be an affine ring over a field. Then, for any $\fp\in \Spec R$, there exists an integer $i\geq 0$ such that the component $T^i_\fp$ of $PE^i(R)$ is non-trivial.  
\end{corollary}

Let $k$ and $s$ be as in Conjecture \ref{Gruson conj}.
Assume that $R$ is an affine ring over $k$ with $n=\dim R$. 
By \cite[II; Corollary 3.3.2]{GR}, the projective dimension of any flat $R$-module is at most $s+1$.
Thus, we see from Remark \ref{minimal pure remark 3} that the pure-injective dimension of $R$ is at most $\inf\{s+1, n\}$, that is, $PE^i(R)=0$ for $i>\inf\{s+1, n\}$. 
See also \cite[II; Corollary 3.2.7]{GR} and \cite[Theorem 8.4.12]{EJ}. 
In particular, when $\inf\{s+1, n\}=0$, {\it i.e.}, $R$ is a $0$-dimensional affine ring, then $R$ is pure-injective.

Now we suppose that $\inf\{s+1, n\}=1$. This means that $k$ is at most countable and $n\geq 1$, or $k$ is any field and $n=1$.
Then the minimal pure-injective resolution of $R$ is of the form
\begin{align*}
0\to \prod_{\fm \in W_0} \widehat{R_\fm}\to \prod_{\fp\in W_{\geq 1}}T^1_\fp\to 0.
\end{align*}
In this case, it follows from Corollary \ref{add remark} that $T^1_\fp\neq 0$ for any $\fp\in W_{\geq 1}$, see also \cite[Theorem 4.13]{PT1}.

Next we consider the case that $\inf\{s+1, n\}\geq 2$. 
Combining Remark \ref{minimal pure remark 3}(iii), Corollary \ref{strong form} and Corollary \ref{add remark}, we can obtain the proposition below.

\begin{proposition}\label{Extension of Enochs}
Let $k$ be an uncountable field and $R$ be an affine ring over $k$ with $\dim R\geq 2$.
Using Remark \ref{minimal pure remark 3}(i),
we write $PE^1(R)= \prod_{\fp\in W_{\geq 1}} T^1_\fp$.
Then it holds that 
$$PE^1(R)=\prod_{\fp\in W_{1}} T^1_\fp,$$
where $T^1_\fp\neq 0$ for all $\fp\in W_1$.
\end{proposition}

This proposition extends Enochs's result \cite[Theorem 3.5]{E}, in which he proved this fact for polynomial rings over the fields of real numbers and complex numbers.

Finally we focus on the case that $\inf\{s+1, n\}=2$.
By Corollary \ref{add remark} and Proposition \ref{Extension of Enochs}, we have the following result.

\begin{theorem}\label{complete description}
Let $k$ be an uncountable field and $R$ be an affine ring over $k$ with $\dim R\geq 2$.
Assume that $|k|=\aleph_1$ or $\dim R=2$.
Using Remark \ref{minimal pure remark 3}(i),
we write $PE^i(R)= \prod_{\fp\in W_{\geq i}} T^i_\fp$. Then, the minimal pure-injective resolution of $R$ is of the form
$$0\to \prod_{\fm \in W_0} \widehat{R_\fm}\to \prod_{\fp\in W_{1}} T^1_\fp\to \prod_{\fp\in W_{\geq 2}}T^2_\fp\to 0,$$
where $T^1_\fp\neq 0$ for all $\fp\in W_{1}$, and $T^2_\fp\neq 0$ for all $\fp\in W_{\geq 2}$.
\end{theorem}

\begin{proof}[Proof of Theorem \ref{partial answer}]
Since $\dim R/\fp\geq 2$, we see that $U(\fp)\cap W_0=\emptyset$, $U(\fp)\cap W_1=\emptyset$ and $U(\fp)\cap W_{\geq 2}=U(\fp)$. 
Therefore, setting $P$ as the minimal pure-injective resolution in Theorem \ref{complete description}, we obtain the following isomorphisms in $\mathrm{D}(R)$;
$$\RHom_R(R_\fp, R)\cong \Hom_R(R_\fp, P)\cong  \prod_{\fq\in U(\fp)}T^2_\fq[-2].$$
Hence $\Ext^0_R(R_\fp, R)=\Ext^1_R(R_\fp, R)=0$ and $\Ext^2_R(R_\fp, R)\cong \prod_{\fq\in U(\fp)}T^2_\fq \neq 0$.
\end{proof}

Suppose that $R$ is a polynomial ring over the field of  real numbers or complex numbers. Then, under the continuum hypothesis, we can apply Theorem \ref{partial answer} to $R$. 
In fact, we need the continuum hypothesis  to ensure that the pure-injective dimension of a flat $R$-module is at most $2$, see \cite[\S 8; p.228]{BLO} and \cite[Theorem 8.4.12]{EJ}.

There is a local version of Conjecture \ref{Gruson conj}; it claims that if $\fm$ is a maximal ideal of $R=k[X_1,\ldots, X_n]$, then $\Ext^i_{R_\fm}(Q(R), R_\fm)\neq 0$ if and only if $i=\inf\{s+1, n\}$ (\cite[\S1; 1]{AT}).
When $\inf\{s+1, n\}\leq 1$, this is true. 
In \cite[Proposition 3.2]{G}, Gruson verified that the equivalence holds if $s\geq 1$, $n=2$ and $\fm =(X_1, X_2)$.

Under the assumption that $\inf\{s+1, n\}\geq 2$, Thorup \cite[Theorem 13]{AT} also showed that $\Ext^1_{R_\fm}(Q(R),R_\fm)=0$, where $\fm$ is any maximal ideal of $R$.
Therefore, provided that $\inf\{s+1, n\}=2$, it is enough to solve Question \ref{question 1} in order to prove the local version.



\bibliographystyle{amsplain}

\end{document}